\documentclass[10pt]{article}
\usepackage{amsmath,amssymb}
\usepackage[pdftex]{graphicx}
\usepackage{fullpage}
\usepackage{cite}
\usepackage{palatino}
\usepackage{amsmath,amsthm,amssymb}

\newcommand{\C}{\mathbb{C}}

\newtheorem{thm}{Theorem}[section]
\newtheorem{lem}[thm]{Lemma}
\newtheorem{prop}[thm]{Proposition}
\newtheorem{define}[thm]{Definition}
\newtheorem{exa}[thm]{Example}
\newtheorem{rem}[thm]{Remark}

\newtheorem{cor}[thm]{Corollary}
\renewcommand{\b}{\beta}

\renewcommand{\d}{\delta}

\newcommand{\0}{\mathbf{0}}
\newcommand{\av}[1]{\left|{#1}\right|}

\newcommand{\ip}[2]{\left\langle{{#1}},{{#2}}\right\rangle}

\renewcommand{\l}{\left}

\renewcommand{\r}{\right}

\newcommand{\be}{\begin{enumerate}}
\newcommand{\bi}{\begin{itemize}}
\newcommand{\ee}{\end{enumerate}}
\newcommand{\ei}{\end{itemize}}
\newcommand{\ii}{\item}

\newcommand{\om}{{\bf 1}}
\newcommand{\Spec}{{\mathrm {Spec}}}

\title{Optimizing Gershgorin for Symmetric Matrices}
\author{Lee DeVille\\Department of Mathematics\\University of Illinois}

\begin{document}

\maketitle

\begin{abstract}
  The Gershgorin Circle Theorem is a well-known and efficient method
  for bounding the eigenvalues of a matrix in terms of its entries.
  If $A$ is a symmetric matrix, by writing $A = B + x{\bf 1}$, where
  ${\bf 1}$ is the matrix with unit entries, we consider the problem
  of choosing $x$ to give the optimal Gershgorin bound on the
  eigenvalues of $B$, which then leads to one-sided bounds on the
  eigenvalues of $A$.  We show that this $x$ can be found by an
  efficient linear program (whose solution can in may cases be written
  in closed form), and we show that for large classes of matrices,
  this shifting method beats all existing piecewise linear or
  quadratic bounds on the eigenvalues.  We also apply this shifting
  paradigm to some nonlinear estimators and show that under certain
  symmetries this also gives rise to a tractable linear program.  As
  an application, we give a novel bound on the lowest eigenvalue of a
  adjacency matrix in terms of the ``top two'' or ``bottom two''
  degrees of the corresponding graph, and study the efficacy of this
  method in obtaining sharp eigenvalue estimates for certain classes
  of matrices.

\end{abstract}

{\bf AMS classification:}  65F15, 15A18, 15A48

{\bf Keywords:} Gershgorin Circle Theorem, Eigenvalue Bounds, Linear Program, Adjacency Matrix

\section{Introduction}\label{sec:intro}

One of the best known bounds for the eigenvalues of a matrix is the
classical Gershgorin Circle Theorem~\cite{Gersh, Varga.book}, which
allows us to determine an inclusion domain of the eigenvalues of a
matrix that can be determined solely from the entries of this matrix.
Specifically, given an $n\times n$ matrix $A$ whose entries are
denoted by $a_{ij}$, we define
\begin{equation*}
  R_i(A) := \sum_{j\neq i}\av{a_{ij}},\quad D_i = \{z\in\C:\av{z-a_{ii}} \le R_i(A)\}.
\end{equation*}
Then the spectrum of $A$ is contained inside the union of the $D_i$.

If $A$ is a symmetric matrix, then the spectrum is real.  In general,
we might like to determine whether $A$ is positive definite, or even
to determine a bound on the spectral gap of $A$, i.e. a number
$\alpha>0$ such that all eigenvalues are greater than $\alpha$.  Let
us refer to the {\em diagonal dominance} of the $i$th row of the
matrix by the quantity $a_{ii} - R_i(A)$.  A restatement of Gershgorin
is that the smallest eigenvalue is at least as large as the smallest
diagonal dominance of any row.

As an example, let us consider the matrix
\begin{equation}\label{eq:56}
  A=\left(\begin{array}{ccc} 6&5&5 \\ 5&6&5 \\ 5&5&6\end{array}\right),
\end{equation}
and we see the quantity $a_{ii}-R_i(A) = -4$ for all rows.  From this
and Gershgorin we know that all eigenvalues are at least $-4$, but we
cannot assert that $A$ is positive definite.  However, we can compute
directly that $\Spec(A) = \{1,1,16\}$; not only is $A$ positive
definite, it has a spectral gap of $1$.  

We then note that $A = I + 5{\bf 1}$ (here and below we use ${\bf 1}$
to denote the matrix with all entries equal to $1$). Since $I$ is
positive definite and ${\bf 1}$ is positive semidefinite, then it
follows that $A$ is positive definite as well.  In fact, we know that
the eigenvalues of $A$ are at least as large as those of $I$, so that
$A$ has a spectral gap of at least $1$, which gives us the exact
spectral gap.  

Let us consider the family of matrices $A - x{\bf 1}$ with $x\ge 0$,
and consider how the diagonal dominance changes. Subtracting off the
value of $x$ from each entry moves the diagonal to the left by $x$,
but as long as the off-diagonal terms are positive, it decreases the
off-diagonal sum by a factor of $2x$. This has the effect of moving
the center of each Gershgorin circle to the left by $x$, but decrease
the radius by $2x$, effectively moving the left-hand edge of each
circle to the right by $x$. This gives an improved estimate for the
lowest eigenvalue.  It is not hard to see in this case that choosing
$x=5$ gives the optimal diagonal dominance of 1.  Generalizing this
idea: let $A$ be a symmetric $n\times n$ matrix with strictly positive
entries. Since ${\bf 1}$ is positive semidefinite, whenever $x\ge0$,
$A = B + x \om$ is positive definite whenever $B$ is.  However,
subtracting $x$ from each entry, at least until some of the entries
become negative, will only increase the diagonal dominance of each
row: each center of each Gershgorin circle moves to the left by $x$,
but the radius decreases by $(N-1)x$.  From this it follows that we
can always improve the Gershgorin bound for such a matrix with
positive entries by choosing $x$ to be the smallest off-diagonal term.

However, now consider the matrix
\begin{equation}\label{eq:Aneg}
  A = \left(\begin{array}{ccc} 6&1&3 \\ 1&7&4 \\3&4&5\end{array}\right).
\end{equation}
We see directly that the bottom row has the smallest diagonal
dominance of $-2$.  Subtracting off a 1 from each entry gives
\begin{equation*}
  A - 1\cdot \om = \left(\begin{array}{ccc} 5&0&2 \\ 0&6&3 \\2&3&4\end{array}\right),
\end{equation*}
and now the bottom row has a diagonal dominance of $-1$.  From here we
are not able to deduce that $A$ is positive definite.  However, as it
turns out, for all $x\in(2,10/3)$, the matrix $A-x\om$ has a positive
diagonal dominance for each row, and by the techniques described in
this paper we can compute that the optimal choice is $x=3$:
\begin{equation*}
  A - 3\cdot \om = \left(\begin{array}{ccc} 3&-2&0 \\ -2&4&1 \\0&1&2\end{array}\right),
\end{equation*}
and each row has a diagonal dominance of $1$.  Since $3\cdot\om$ is
positive semi-definite, this guarantees that $A$ is positive definite
and, in fact, has a spectral gap of at least $1$.  In fact, we can
compute directly that the eigenvalues of $A$ are $(11.4704, 5.39938,
1.13026)$, so this it not too bad of a bound for this case.

At first it might be counter-intuitive that we should subtract off
entries and cause some off-diagonal entries to become negative, since
the $R_i(A)$ is a function of the absolute value of the off-diagonal
terms and making these entries more negative can actually make the
bound from the $i$th row worse.  However, there is a coupling between
the rows of the matrix, since we are looking for the minimal diagonal
dominance across all of the rows: it might help us to make some rows
worse, as long as we make the worst row better.  In this example,
increasing $x$ past $1$ does worsen the diagonal dominance of the
first row, but it continues to increase the dominance of the third
row, and thus the minimum is increasing as we increase $x$.  Only
until the rows pass each other should we stop, and this happens at
$x=3$.  From this argument it is also clear that the optimal value of
$x$ must occur when two rows have an equal diagonal dominance.

The topic of this paper is to study the possibility of using the idea
of shifting the entries of the matrix uniformly to obtaining better
eigenvalue bounds.  Note that this technique is potentially applicable
to any eigenvalue bounding method that is a function of the entries of
the matrix, and we consider examples other than Gershgorin as well. We
show that using Gershgorin, which is piecewise linear in the
coefficients of the matrix, the technique boils down to the
simultaneous optimization of families of piecewise linear functions,
so can always be solved, e.g. by conversion to a linear program. We
will also consider ``nonlinear'' versions of Gershgorin theorems and
shift these as well, but we see that for these cases we can lose
convexity and in general this problems are difficult to solve.

We show below that for many classes of matrices, the ``shifted
Gershgorin'' gives significantly better bounds than many of the
unshifted nonlinear methods that are currently known.

\section{Notation and background}\label{sec:bg}

%We start by giving some definitions and lemmas.  
Throughout the
remainder of this paper, all matrices are symmetric.

\newcommand{\method}[1]{{\mathsf{{{#1}}}}}
\newcommand{\lb}[2]{\lambda_{\method{{#1}}}({#2})}
\newcommand{\rb}[2]{\rho_{\method{{#1}}}({#2})}
\newcommand{\ls}[2]{\lambda^{\mathsf{S}}_{\method{{#1}}}({#2})}
\newcommand{\rs}[2]{\rho^{\mathsf{S}}_{\method{{#1}}}({#2})}

\begin{define}
  Let $\method{Q}$ be a method that gives upper and lower bounds on
  the eigenvalues of a matrix as a function of its entries.  We denote
  by $\lb Q A$ as the greatest lower bound method $\method{Q}$ can
  give on the spectrum of $A$, and by $\rb Q A$ the least upper bound.
\end{define}

Here we review several known methods.  There are many methods known
that we do not mention here; a very comprehensive and engaging review
of many methods and their histories is contained in~\cite{Varga.book}.

\be

\ii {\bf Gershgorin's Circles.}~\cite{Gersh} As mentioned in the
introduction, define

\begin{equation*}
  R_i(A) = \sum_{j\neq i} \av{a_{ij}},\quad D_i = \{z\in\C:\av{z-a_{ii}} < R_i(A)\}.
\end{equation*}
Then $\Spec(A) \subseteq \cup_i D_i$.  For this method,
we obtain the bounds
\begin{equation}\label{eq:lrGersh}
  \lb G A  =   \min_i (a_{ii} - R_i(A)),\quad \rb G A = \max_i (a_{ii} + R_i(A)).
\end{equation}

\ii {\bf Brauer's Ovals of Cassini.}~\cite{Ostrowski.37},~\cite[eqn. (21)]{Brauer.47} Using the
definitions of $R_i(A)$ as above, let us define
\begin{equation*}
  B_{ij} = \{z\in\C:\av{z-a_{ii}}\av{z-a_{jj}} \le R_i(A)R_j(A)\}.
\end{equation*}
Then $\Spec(A) \subseteq \cup_{i\neq j} B_{i,j}$.  It can be
shown~\cite[Theorem 2.3]{Varga.book} that this method always gives a
nonworse bound than Gershgorin (note for example that choosing $i=j$
gives a Gershgorin disk).  The bounds come from the left- or
right-hand edges of this domain, so that $z-a_{ii},z-a_{jj}$ have the
same sign.  This gives
\begin{equation*}
  a_{ii}a_{jj} - R_i(A)R_j(A) - (a_{ii}+a_{jj}) z + z^2 \le 0.
\end{equation*}
The roots of this quadratic are
\begin{equation*}
  \frac{a_{ii}+a_{jj}}2 \pm \sqrt{(a_{ii}-a_{jj})^2 + R_i(A)R_j(A)}.
\end{equation*}
Therefore we have

\begin{align}
  \lb B A &= \min_{i\neq j }\l(\frac{a_{ii}+a_{jj}}2 - \sqrt{(a_{ii}-a_{jj})^2 + R_i(A)R_j(A)}\r),\label{eq:lBrauer}\\
  \rb B A &= \max_{i\neq j }\l(\frac{a_{ii}+a_{jj}}2 + \sqrt{(a_{ii}-a_{jj})^2 + R_i(A)R_j(A)}\r),\label{eq:rBrauer}
\end{align}

\ii {\bf Melman method.}  According to Theorem~2 of~\cite{Melman.10}, if we define
\begin{equation*}
  \Omega_{ij}(A)
  = \{z\in\C:\av{(z-a_{ii})(z-a_{jj}) - a_{ij}a_{ji}} \le \av{z-a_{jj}}\sum_{k\neq i,j}\av{a_{ik}} + \av{a_{ij}}\sum_{k\neq i,j}\av{a_{jk}}\},
\end{equation*}
then 
\begin{equation*}
  \Spec(A) \subseteq \bigcup_{i=1}^n \bigcap_{j\neq i}\Omega_{ij}(A).
\end{equation*}
From this we can obtain bounds $\ls M A, \rs M A$ similar to those
in~(\ref{eq:lBrauer},~\ref{eq:rBrauer}).

\ii {\bf Cvetkovic--Kostic--Varga method.}  Let $S$ be a nonempty set
of $[n]$, and $\overline{S} = [n]\setminus S$, and let 
\begin{equation*}
  R_i^S(A) = \sum_{j\in S, j\neq i}\av{a_{ij}},\quad R_i^{\overline S}(A) = \sum_{j\not\in S, j\neq i}\av{a_{ij}}.
\end{equation*}
Further define
\begin{align*}
  \Gamma_i^S(A)&= \l\{z\in\C:\av{z-a_{ii}}\le R_i^S(A)\r\},\\V_{ij}^S(A) &= \l\{z\in C:(\av{z-a_{ii}}-R_i^S(A))(\av{z-a_{jj}}-R_j^{\overline S}(A)) \le R_i^{\overline{S}}(A) R_j^S(A)\r\},
\end{align*}

According to Theorem~6 of~\cite{Cvetkovic.Kostic.Varga.04},
\begin{equation*}
  \Spec(A) \subseteq \l(\bigcup_{i\in S} \Gamma_i^S\r)\cup \l(\bigcup_{i\in S,j\in\overline{S}}V_{ij}^S(A)\r).
\end{equation*}
\ee

\begin{define}
  Let $\method{Q}$ be a method as above, and let us define the {\bf
    shifted-$\method{Q}$ method} given by
\begin{equation*}
  \ls QA = \sup_{x\ge 0} \lb Q{A-x\om},\quad  \rs QA = \inf_{x\le 0} \rb Q{A-x\om}.
\end{equation*}
Note that the domain of optimization is $x\ge0$ in one definition and
$x\le 0$ in the other.  For example $\ls Q A$ is the best lower bound
that we can obtain by method $\method{Q}$ on the family given by
subtracting a positive number from the entries of $A$, whereas $\rs Q
A$ is the best upper bound obtained by adding a positive number.
\end{define}

\begin{rem}
  We will focus our study on the first two methods in the interest of
  brevity, i.e. we consider $\method{Q} =\method{G,B}$ for Gershgorin
  and Brauer.  However, we present all four above as representative of
  a truly large class of methods that share a common feature: the
  bounds are always a function of the diagonal entries and a sum of
  (perhaps a subset) of the absolute values of the off-diagonal terms
  on each row.  As we show below in Section~\ref{sec:Gershthm}, the
  shifted bounds are a function not just of sums of off-diagonal
  terms, but a function of each individual entry.  As such, it is not
  surprising that shifted bounds can in some cases vastly outperform
  the unshifted bounds --- we will see even that shifted Gershgorin
  can outperform even unshifted nonlinear bounds such as Brauer.
\end{rem}

\begin{lem}\label{lem:shift}
  If $\method{Q}$ is a method that bounds the spectrum of a matrix,
  then the spectrum of $A$ is contained in the interval $[\ls Q A,\rs
  Q A]$, i.e. the scaled bounds are good upper and lower bounds for
  the eigenvalues.
\end{lem}

\begin{proof}
  The proof of this follows from the Courant Minimax Theorem~\cite[\S
  4.2]{Horn.Johnson.book},~\cite[\S 4]{Courant.Hilbert.book}: Let $A$
  be a symmetric $n\times n$ matrix, and let $\lambda_1(A)\le
  \lambda_2(A)\le\dots\le \lambda_n(A)$ be the eigenvalues, then we
  have
  \begin{equation*}
    \lambda_1(A) = \min_{v\neq 0}\frac{\ip{Av}v}{\ip vv},\quad   \lambda_n(A) = \max_{v\neq 0}\frac{\ip{Av}v}{\ip vv}.
  \end{equation*}
  If we write $A = B + x\om$, then
  \begin{equation}\label{eq:l1a}
    \lambda_1(A) = \min_{v\neq 0}\frac{\ip{(B+x\om)v}v}{\ip vv} \ge \min_{v\neq 0}\frac{\ip{Bv}v}{\ip vv} + \min_{v\neq 0}\frac{\ip{x\om v}v}{\ip vv}.
  \end{equation}
  If $x\ge 0$, then the second term is zero (in fact, $\om$ is a
  rank-one matrix with top eigenvalue $n$), and thus we have
  $\lambda_1(A) \ge \lambda_1(B)$.  Since $B = A-x\om$, we have for all $x\ge 0$,
  \begin{equation*}
    \lambda_1(A) \ge \lambda_1(A-x\om) \ge \lb Q{A-x\om},
  \end{equation*}
  and the same is true after taking the supremum.  Similarly,
  \begin{equation*}
    \lambda_n(A) \le \max_{v\neq 0}\frac{\ip{Bv}v}{\ip vv}+ \max_{v\neq 0}\frac{\ip{x\om v}v}{\ip vv},
  \end{equation*}
  and if $x\le 0$, the second term is zero.  Therefore, for all $x\le 0$, 
\begin{equation*}
  \lambda_n(A) \le \lambda_n(A-x\om) \le \rb Q {A-x\om},
\end{equation*}
and the same remains true after taking the infimum of the right-hand side.
\end{proof}

It is clear from the definition that $\ls Q A \ge \lb Q A$ and $\rs Q
A \le \rb Q A$ for any method $\method Q$ and any matrix $A$.  In
short, shifting cannot give a worse estimate, and by
Lemma~\ref{lem:shift} this improved estimate is still a good bound.
Moreover, as long as
\begin{equation*}
  \l.\l(\frac{d}{dx} \lb Q {A-x\om}\r)\r|_{x=0+} > 0,
\end{equation*}
then $\ls Q A > \lb Q A$, and similarly for the upper bounds.
Moreover, we show in many contexts that the bounding functions are
convex in $x$, so the converse is also true in many cases.

\begin{rem}\label{rem:wrongway}
  If we re\"{e}xamine the proof of Lemma~\ref{lem:shift}, we might
  hope that we could obtain information on a larger domain of shifts;
  for example, in~\eqref{eq:l1a}, when $x<0$, we have $\min_{x\neq 0}
  \ip{x\om v }v = xN\ip vv$, so we have for $x\le 0$,
  \begin{equation}\label{eq:xle0}
    \lambda_1(A) \ge \lambda_1(A-x\om) + xN.
  \end{equation}
  We might hope that this could give useful information even when
  $x\le 0$, but some calculation can show that (at least for all of
  the methods described in this paper) the first term on the
  right-hand side cannot grow fast enough to make the right-hand side
  of this inequality increase as $x$ decreases, so this will always
  give a worse bound than just restricting to $x\ge 0$.  For example,
  for Gershgorin, the absolute best case is that all of the signs
  align and the edge of the Gershgorin disk is moving to the right at
  rate $Nx$, but this is exactly counteracted by the second term on
  the right-hand side of~\eqref{eq:xle0}.
\end{rem}

\section{Shifted Gershgorin bounds}\label{sec:Gershthm}

In this section, we concentrate on the shifted Gershgorin method.
Recall that the standard Gershgorin bounds are given by
\begin{equation*}
  \lb G A  =   \min_i (a_{ii} - R_i(A)),\quad \rb G A = \max_i (a_{ii} + R_i(A)).
\end{equation*}
Let us define
\begin{equation}\label{eq:defofds}
  d_i(A,x) = a_{ii} - x -\sum_{j\neq i} \av{a_{ij}-x}, \quad   D_i(A,x) = a_{ii} - x +\sum_{j\neq i} \av{a_{ij}-x},
\end{equation}
then 
\begin{equation*}
  \ls G A = \sup_{x\ge 0}\min_i d_i(A,x),\quad \rs G A = \inf_{x\le 0}\max_i D_i(A,x).
\end{equation*}

\begin{thm} \label{thm:gersh} 
 The main results for the shifted Gershgorin are as follows:

\bi

\ii {\bf (Local Improvement.)}  Iff, for every $i$ with $d_i(A,0) =
\lb GA$, row $i$ of the matrix $A$ has at least $\lfloor n/2-1\rfloor$
positive numbers, then $\ls GA>\lb GA$.  Iff, for every $i$ with
$D_i(A,0) = \rb GA$, row $i$ of the matrix $A$ has at least $\lfloor
n/2-1\rfloor$ negative numbers, then $\rs GA< \rb GA$.

\ii {\bf (Global bounds.)}  Each of the functions $\lb G{A-x\om}, \rb
G{A-x\om}$ can be written as a single piecewise-linear function.
Alternatively, we can write these functions as cut out by $n$ lines,
i.e.
\begin{equation}\label{eq:singlefunction}
  \lb G {A-x\om} = \min_{k=1}^n (r_k x + s_k),\quad \rb G {A-x\om} = \max_{k=1}^n (R_kx + S_k),
\end{equation}
where the constants in the above expressions are given in Definition~\ref{def:rsk};

\ii {\bf (Convexity.)}  It follows from the previous that the
functions $-\lb G {A-x\om}$ and $\rb G {A-x\om}$ are convex, i.e. $\lb
G {A-x\om}$ is ``concave down'' and $\rb G{A-x\om}$ is ``concave up''
in $x$.  From this it follows that the minimizer is attained at a
unique value, or perhaps on a unique closed interval.

\ei

\end{thm}

\begin{rem}\label{rem:Gersh}
  The local improvement part of the theorem above is basically telling
  us that we have to have enough terms of the correct sign to allow
  the shifting to improve the estimates.  If a row has many more
  positive terms than negative terms, for example, then subtracting
  the same number from each entry improves the Gershgorin bound, since
  it decreases the off-diagonal sum --- from this we see that we would
  want the row(s) which are the limiting factor in the $\lb GA$
  calculation to have enough positive terms.  In particular, it
  follows that if there are enough terms of both signs, shifting can
  improve both sides of the bound.

  One quick corollary of this theorem is that if we have a matrix with
  all positive entries (or even all positive off-diagonal entries),
  then shifting is guaranteed to improve the lower bound, yet cannot
  improve the upper bound.  In this sense, our theorem is in a sense a
  counterpoint to the Perron--Frobenius theorem: the bound on the
  spectral radius of a positive matrix is exactly the one given by
  Perron--Frobenius, but we always obtain a good bound on the least
  eigenvalue.

  The global bound part of this theorem tells us that we can write the
  objective function as a single piecewise-linear function, or as an
  extremum of at most $n$ different linear functions.  This more or
  less allows us to write down the optimal bound in closed form, see
  Corollary~\ref{lem:doublemin} below.  In any case, the optimal bound
  can be obtained by a linear program.
\end{rem}

\begin{define}\label{def:rsk}
  Let $A$ be an $n\times n$ symmetric matrix.  We define $r_k := n-2k,
  R_k= 2k-n-2$ for $k=1,\dots,n$.  For each $i=1,\dots,n$, we
  define $\delta_{i,k}$ as follows: let $y_1 \le y_2\le \dots\le
  y_{n-1}$ be the off-diagonal terms of row $i$ of the matrix $A$ in
  nondecreasing order, and define, for $i,k=1,\dots,n$,
  \begin{equation*}
    \delta_{i,k} :=   \sum_{j < k} y_j - \sum_{j\ge k} y_j,
  \end{equation*}
  and then 
  \begin{equation*}
    s_{i,k} = a_{ii} + \delta_{i,k},\quad S_{i,k} = a_{ii} - \delta_{i,k}.
  \end{equation*}

  Finally, we define
  \begin{equation*}
    s_k = \min_i s_{i,k},\quad S_k = \max_i S_{i,k}.
  \end{equation*}

\end{define}

\begin{lem}\label{lem:dD}
  The functions $d_i(A,x), D_i(A,x)$ are piecewise linear, and can be
  written as a minimizer of a family of linear functions:
\begin{equation*}
  d_i(A,x) = \min_k (r_k x  + s_{i,k}),\quad D_i(A,x) = \max_k (R_kx + S_{i,k}).
\end{equation*}
\end{lem}
\begin{proof}
  Define $y$ as in Definition~\ref{def:rsk}. First note that we can
  write
\begin{equation*}
  d_i(A,x) = a_{ii} - x - \sum_{\ell=1}^{n-1} \av{y_\ell-x},
\end{equation*}
since the ordering in the sum does not matter.  Now, note that if
$x\in(y_{k-1},y_k)$ (which could be an empty interval), we have
\begin{align*}
  d_i(A,x) 
  &= a_{ii}-x - \sum_{\ell =1 }^{k-1}(x-y_\ell) - \sum_{\ell = k}^{n-1} (y_\ell-x)\\
  &= a_{ii} + \sum_{\ell=1}^{k-1}y_\ell - \sum_{\ell = k}^{n-1} y_\ell - x - (k-1)x + (n-k)x\\
  &= s_{i,k} + (n-2k)x = s_{i,k} + r_kx.
\end{align*}
In the case where $y_{k-1}=y_k$, it is easy to see that this equality
does not hold on an interval, but it does hold at the point $y_k$.
Finally, noting that $\av{x} \ge x$ and $\av x \ge -x$, this means
that $d_i(A,x)\le r_kx+s_{i,k}$ for all $x$ (note the negative sign in
front of the sum!).  Thus the family $r_kx + s_{i,k}$ dominates
$d_i(A,x)$, and coincides with it on a nonempty set, and this proves
the result.

The argument is similar for $D_i$.  If we write
\begin{align*}
  D_i(A,x) 
  &= a_{ii}-x + \sum_{\ell =1 }^{k-1}(x-y_\ell) + \sum_{\ell = k}^{n-1} (y_\ell-x)\\
  &= a_{ii} - \sum_{\ell=1}^{k-1}y_\ell + \sum_{\ell = k}^{n-1} y_\ell - x + (k-1)x - (n-k)x\\
  &= S_{i,k} + (2k-n-2)x = S_{i,k} + R_kx,
\end{align*}
and the remainder follows by taking maxima.
\end{proof}

{\noindent \bf Proof of Theorem~\ref{thm:gersh}.}

First we prove the ``global bounds'' part of the theorem. From  Lemma~\ref{lem:dD},
\begin{equation*}
  \ls G A = \sup_{x\ge0} \min_i d_i(A,x) = \sup_{x\ge0} \min_i \min_k (r_k x + s_{i,k} ) = \sup_{x\ge0}\min_k (r_k x + \min_i s_{i,k}) = \sup_{x\ge0}\min_k (r_k x +s_k),
\end{equation*}
and we are done.  The proof is the same for $\rs G A$.

The argument for convexity is similar.  We have
\begin{equation*}
  \lb G {A-x\om}= \min_i d_i(A,x)=\min_i \min_k (r_k x + s_{i,k} ). 
\end{equation*}
As a minimizer of linear functions, it follows that $\lb G{A-x\om}$ is
concave down, and similarly, as $\rb G{A-x\om}$ is the maximizer of
linear functions, it is concave up.

Finally, we consider the ``local improvement'' statement.  Choose and
fix a row $i$, and define $y$ as in Definition~\ref{def:rsk}.  If all
of the entries of $y_k$ are positive, then for $x < y_1$, $d_i(A,x) =
(n-2)x + \sum_k y_k$, and thus has a positive slope at $x=0$.
Otherwise, let us write $y_{k-1} \le 0 \le y_{k}$.  If $n-2k >0$, then
$d_i(A,x)$ has a positive slope at $x=0$.  Thus, as long as
$y_{\lceil{n/2}\rceil}>0$, then $d_i(A,x)$ has a positive slope at
$0$.  In short, $d_i(A,x)$ has a positive slope at zero iff at least
$n-1-\lceil{n/2}\rceil = \lfloor{n/2-1}\rfloor$ of the entries of $y$
are positive.  Since $\lb G{A-x\om}$ is the minimum of the $d_i$, as
long as this is true for every row which minimizes the quantity in a
neighborhood of zero, then $\lb G{A-x\om}$ is itself increasing at
$x=0$.

\qed

\begin{cor}\label{lem:doublemin}
  Let $A$ be an $n\times n$ symmetric matrix and consider a set of pairs of indices defined by
\begin{equation*}
  Q = \{(k,l): k \le n/2, l > n/2, s_l \ge s_k\}.
\end{equation*}
  Then if $Q\neq\emptyset$, then $\ls G A > \lb G A$ and 
\begin{equation}\label{eq:bkbl}
  \ls G A = \min_{(k,l)\in Q} \l(\frac{(n/2-k)s_l + (l-n/2)s_k}{l-k}\r),
\end{equation}
and if $Q=\emptyset$ then $\ls G A = \lb G A$.  Moreover, the minimum
of $\lb G {A-x\om}$ occurs at $x = (s_l-s_k)/(2(l-k))$, where $k,l$
are chosen so that~\eqref{eq:bkbl} is minimized.  Also, if the
(off-diagonal) terms of $A$ are all positive, then the set $Q$ is the
set $\{(k,l):k\le n/2, l> n/2\}$.
\end{cor}

\begin{proof}
  First note that if the entries of $A$ are positive, then $s_{i,k} <
  s_{i,k+1}$ for all $i$, and thus $s_k < s_{k+1}$.  From this it
  follows that if $k\le l$ then $s_l\ge s_k$, and the last sentence
  follows.

  From Theorem~\ref{thm:gersh}, in
  particular~\eqref{eq:singlefunction}, it follows that the supremum
  of $\lb G{A-x\om}$ will be attained at the intersection of two
  lines, and moreover it is clear that this must between between a
  line of nonnegative slope ($k\le n/2$) and a line of nonpositive
  slope ($k\ge n/2$).  Finding the intersection point gives
\begin{equation*}
  x = \frac{s_l-s_k}{r_k-r_l}.
\end{equation*}
Since the denominator is positive, we need the numerator be
nonnegative, which gives the condition $s_l \ge s_k$.  Plugging this
point into the line $s_k x + r_k$ gives the expression in parentheses
in~\eqref{eq:bkbl}.  The minimum such intersection gives us our
answer.
\end{proof}

\begin{exa}
  Let $A$ be a $3\times3$ matrix and compute $s_1,s_2,s_3$.  Using
  Corollary~\ref{lem:doublemin}, there are at most only $2 = 1\times
  2$ terms to minimize over: $k=1,l=2$ and $k=1,l=3$.  If $s_1$ is the
  largest, then we cannot improve the lower bound by shifting, and in
  fact we obtain $\lb G A = \ls G A = \min(s_2,s_3)$.  If $s_1$ is the
  smallest (which will always happen, for example, when the
  off-diagonal entries of $A$ are positive), then we obtain
  \begin{equation*}
    \ls G A = \min\l(\frac{s_1+s_2}2,\frac{3s_1 + s_3}4\r).
  \end{equation*}
  If we consider $A$ from~\eqref{eq:56}, since $s_{i,k}$ is independent
of $i$, we can obtain $s_k$ from any row, and we have
\begin{equation*}
  s_1 = 6-5-5=-4,\quad s_2 = 6+5-5 = 6,\quad s_3 = 6+5+5 = 11,
\end{equation*}
and thus
\begin{equation*}
  \ls G A = \min(1, 29/4) = 1.
\end{equation*}
Since the minimum was attained by the $(k=1,l=2)$ pair, we have that
it is attained at $x=(6+4)/2 = 5$.  This matches the analysis in the
introduction.
\end{exa}

\begin{exa}
Choosing $A$ as in~\eqref{eq:Aneg}, we have
\begin{equation*}
  s_{1,k} = \{2,4,10\},\quad s_{2,k} = \{2,4,10\},\quad s_{3,k} = \{-2,6,12\},
\end{equation*}
so that 
\begin{equation*}
  s_1 = -2,\quad s_2 = 4,\quad s_3 = 10.
\end{equation*}
Therefore
\begin{equation*}
  \ls G A = \min(1,1) = 1.
\end{equation*}
We also see that we can choose either pair to find the $x$ at which
this is minimized, using the $(k=1,l=3)$ pair gives at $x =
(10+2)/(2\cdot 2) = 3$.  We could also ask for which $x$ the matrix
$A-x\om$ is guaranteed to be positive definite; note that we have

\begin{equation*}
  \lb G{A-x\om} = \min\{x-2,-x+4,-3x+10\},
\end{equation*}
so it is in the domain $(2,\infty)\cap(-\infty,4)\cap(-\infty,10/3) = (2,10/3)$.

\end{exa}

\section{Shifted Brauer Bounds}\label{sec:nlthm}

In this section we consider the shifted Brauer estimator.  The Brauer
estimate is nonlinear and, moreover, we show that the convexity
properties of shifted Gershgorin no longer apply.  We show that this
adds considerable technical difficulty to the problem, except in
certain cases.  Although we do not study the shift of some of the
other nonlinear estimators from Section~\ref{sec:bg}, we see from the
difficulties that arise for Brauer that it is unlikely for us to
obtain a tractable optimization problem without further assumptions.

We recall from~\eqref{eq:lBrauer} that
\begin{equation*}
    \lb B A = \min_{i\neq j }\l(\frac{a_{ii}+a_{jj}}2 - \sqrt{(a_{ii}-a_{jj})^2 + R_i(A)R_j(A)}\r),
\end{equation*}
and thus
\begin{equation*}
  \ls B A = \inf_{x\ge0}\min_{i\neq j }\l(\frac{a_{ii}+a_{jj}}2 - x -\sqrt{(a_{ii}-a_{jj})^2 + R_i(A-x\om)R_j(A-x\om)}\r).
\end{equation*}
Noting that
\begin{equation*}
  R_i(A-x\om) = \sum_{k\neq i} \av{a_{ik} - x},
\end{equation*}
and writing
\begin{equation*}
  f_{ij}(A,x) = \frac{a_{ii}+a_{jj}}2 - x -\sqrt{(a_{ii}-a_{jj})^2 + \l(\sum_{k\neq i}\av{a_{ik}-x}\r)\l(\sum_{k\neq j} \av{a_{jk}-x}\r)},
\end{equation*}
this simplifies to
\begin{equation}\label{eq:lBshift}
    \ls B A = \inf_{x\ge0}\min_{i\neq j }f_{ij}(A,x).
\end{equation}
This last expression is deceptively complicated, because in fact the
functions $f_{ij}(A,x)$ are not convex, or even very simple.

\begin{exa}
  In Figure~\ref{fig:tworow} we plot a few examples of $f_{ij}(A,x)$
  to get a sense of the shape of such functions.  In each case, we
  specify the diagonal entry, and the set of off-diagonal entries, of
  the first two rows of a matrix, and plot $f_{12}(A,x)$.  For
  example, if
\begin{equation*}
  A_1 = \left(\begin{array}{ccc} 0&1&2 \\ 4&0&5 \\?&?&?\end{array}\right),\quad   A_2 = \left(\begin{array}{ccc} 0&1&4 \\ 2&0&5 \\?&?&?\end{array}\right),
\end{equation*}
we obtain the figures in Figure~\ref{fig:tworow}.  (Note that as long
as we choose the diagonal entries the same, this simply shifts the
function without changing its shape.)  First note that neither
function has the properties from the last section: although the
function in the right frame is close to being concave down, it is not
quite, and clearly the function in the left frame is far from concave
down.

\begin{figure}[ht]
\begin{centering}
  \includegraphics[width=.45\textwidth]{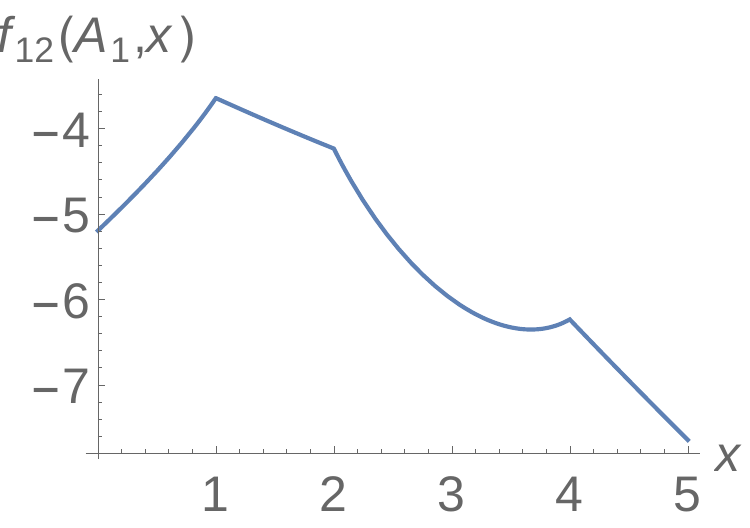}
  \includegraphics[width=.45\textwidth]{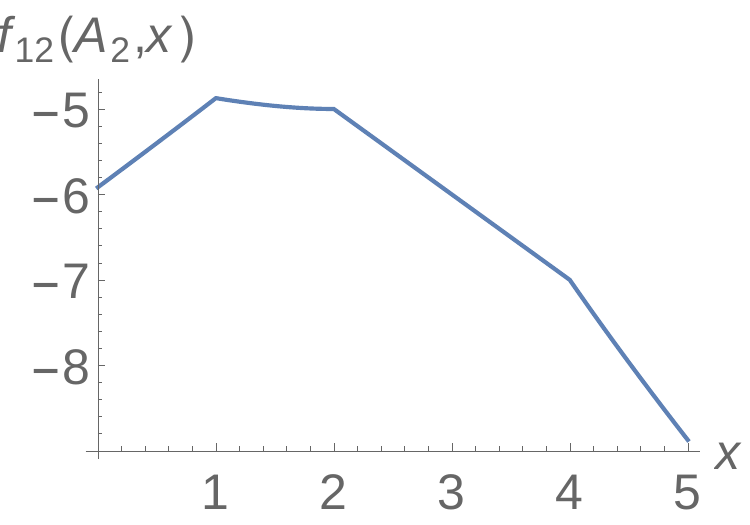}
  \caption{Two plots of $f_{12}(A,x)$ for various choices of entries.}
  \label{fig:tworow}
\end{centering}
\end{figure}

\end{exa}

We could, in theory, write out the expression~\eqref{eq:lBshift} as a
single minimization over a large number of functions.  Consider the
expression for $f_{ij}(A,x)$: the expression under the radical is
piecewise quadratic on any interval between distinct off-diagonal
terms of the two rows, and since $\av x \ge x,-x$, this means that
$f_{ij}(A,x)$ could be written as the minimum of such terms.  However,
this approach will not be as fruitful as it was in the previous
section, since these functions can be both concave up and concave
down.  Moreover, since the bounding functions are nonlinear with a
square root, it is possible that they have a domain strictly smaller
than $x\ge 0$.  In any case, extrema do not have to occur at the
places where the domain of definition shifts.

\newcommand{\ftilde}{d}

However, there is one case where an approximation gives a tractable
minimization problem, which we state below.

\begin{thm}\label{thm:halfGersh}
  Assume that all of the diagonal entries of $A$ are the same, and
  denote this number by $q$.  Then if we define
\begin{equation}\label{eq:tildel}
  \widetilde{\lambda^{\mathsf{S}}_B}(A) = \sup_{x\ge 0}\min_{i\neq j}\left( q - x - \frac12\l(R_i(A-x\om)+R_j(A-x\om)\r)\right),
\end{equation}
then $\widetilde{\lambda^{\mathsf{S}}_B}(A)$ is still a good lower bound for the eigenvalues of $A$.
\end{thm}  

\begin{rem}
  Notice that the RHS of~\eqref{eq:tildel} is simply the average of
  the Gershgorin bounds given by the $i$th and $j$th rows.  As such,
  these functions are again piecewise linear and can be attacked in a
  similar manner to that of the previous section, which we discuss
  below.  Also, notice in the statement of the theorem that we can
  assume wlog that $q=0$, since adding $q$ to every diagonal entry
  simply shifts every eigenvalue by exactly $q$.
\end{rem}

\begin{proof}
  If $a_{ii} = a_{jj}$, then
\begin{align*}
  f_{ij}(A,x) &= a_{ii} - x -\sqrt{R_i(A-x\om)R_j(A-x\om)} \\&\ge a_{ii} - x - \frac12(R_i(A-x\om)+R_j(A-x\om)) =: \ftilde_{ij}(A,x).
\end{align*}
Notice that 
\begin{equation*}
  \ftilde_{ij}(A,x) = \frac12 (d_i(A,x) + d_j(A,x)),
\end{equation*}
where the $d_i$ are the linear bounding functions
from~\eqref{eq:defofds}.  Thus we can define
\begin{equation*}
  \widetilde{\lambda^{\mathsf{S}}_B}(A):= \inf_{x\ge 0}\min_{i\neq j} \ftilde_{ij}(A,x),
\end{equation*}
and we are guaranteed that $\widetilde{\lambda^{\mathsf{S}}_B}(A) \le
\ls B A$, so that it is still a good bound for the eigenvalues.
\end{proof}

\begin{define}
  Let $A$ be an $n\times n$ symmetric matrix with all diagonal entries
  equal to $q$.  Define $p_k = n-1-k$ for $k=1,\dots 2n-1$.  For each
  $i\neq j$, let us define $\d_{i,j,k}$ and $q_{i,j,k}$ as follows:
  let $y_1\le y_2\le\dots\le y_{2n-2}$ be the nondiagonal entries of
  rows $i$ and $j$ of the matrix $A$ in nondecreasing order, and then
  for $k=1,\dots, 2n-1$,
\begin{equation*}
  \d_{i,j,k} = \sum_{m<k} y_m - \sum_{m\ge k}y_m,
\end{equation*}
 $q_{i,j,k} = q + \d_{i,j,k}$ and finally $q_{k} = \min_{i\neq j}
q_{i,j,k}$.
\end{define}

\begin{prop}
  If the diagonal entries of $A$ are all the same, then (q.v.~Theorem~\ref{thm:gersh})
\begin{equation*}
  \lb B {A-x\om} = \min_{k=1}^{2n-1} (p_k x + q_k).
\end{equation*}
\end{prop}

\begin{proof}
  The proof is similar to that of the previous section.  The main
  difference is that while the slope $d_i(A,x)$ decreases by two every
  time $x$ passes through a non-diagonal entry of row $i$, the slope
  of $d_{ij}(A,x)$ decreases by one every time $x$ passes through a
  non-diagonal entry of row $i$ or row $j$.  From this the potential
  slopes of $d_{ij}(A,x)$ can be any integer between $n-2$ and $-n$,
  and everything else follows. 
\end{proof}

\section{Applications and Comparison}

In Section~\ref{sec:adj} we use the main results of this paper to give
a bound on the lowest eigenvalue of the adjacency matrix of an
undirected graph, which is sharp for some circulant graphs.  In
Section~\ref{sec:randomnumerics} we perform some numerics to compare
the shifted Gershgorin method to other methods on a class of random
matrices with nonnegative coefficients.  Finally in
Sections~\ref{sec:pdnumerics} and~\ref{sec:spread} we consider the
efficacy of shifting as a function of many of the entries of the
matrix.

\subsection{Adjacency matrix of a graph}\label{sec:adj}

\begin{define}
Let $G=(V,E)$ be an undirected loop-free graph; $V=[n]$ is the set of
vertices of the graph and $E\subseteq V\times V$ is the set of edges.
We define the {\em adjacency matrix} of the graph to be the $n\times
n$ matrix $A(G)$ where
\begin{equation*}
  A(G)_{ij}  =\begin{cases} 1,&(i,j)\in E,\\ 0,&(i,j)\not\in E.\end{cases}
\end{equation*}
We also write $i\sim j$ if $(i,j)\in E$.  We define the {\bf degree}
of vertex $i$, denoted $\deg(i)$, to be the number of vertices
adjacent to vertex $i$.  We will denote the maximal and minimal
degrees of $G$ by $\Delta(G), \delta(G)$.  We also denote
$\Delta^2(G)$ (resp.~$\delta^2(G)$) as the second largest
(resp.~smallest) degrees of $G$.
\end{define}

We use the convention here that $A(G)_{ii} = 0$ for all $i$.  This
choice varies in the literature but making these diagonal entries to
be one instead will only shift every eigenvalue by $1$.  From
Remark~\ref{rem:Gersh}, since all off-diagonal entries are
nonnegative, we have $\rb G A = \rs G A$.  However, we can improve the
left-hand bound by scaling.  Intuitively, we expect the scaling to
help when there are many positive off-diagonal terms, so we expect
this to help best when we have a ``dense'' graph, i.e. a graph with
large $\delta(G)$.

We can use the formulas from above, but in fact we can considerably
simplify this computation.  Note that
\begin{equation*}
  d_i(A-x\om) = -x - \sum_{j\sim i} \av{1-x} - \sum_{j\not\sim i}\av{x}.
\end{equation*}
Notice that this function is decreasing for $x>1$ in any case, so we
can restrict the domain of consideration to $x\in[0,1]$.  Restricted
to this domain, the function simplifies to
\begin{align*}
  d_i(A-x\om) &= -x - \deg(i)(1-x) - (n-1-\deg(i))(x)\\& = -\deg(i) + (2\deg(i)-n)x.
\end{align*}
Noting that $-\deg(i) + (2\deg(i)-n)x = -nx + \deg(i)(2x-1)$, we see
that all of these functions are equal to $-n/2$ at $x=1/2$.  From this
it follows that the only possible optimal points are $x=0,1/2,1$ in
the three cases that all of the slopes are negative, some are of each
sign, and all slopes are positive.  (The simplest way to see this is
to note that the family of lines given by $d_i$ all ``pivot'' around
the point $(1/2,-n/2)$.)

The case where all slopes are negative is if $2\deg(i)-n<0$ for all
$i$, or $\Delta(G) < n/2$.  In this case, the unshifted Gershgorin
bound of $-\Delta(G)$ is the best that we can do.  If not all of the
slopes are of one sign, i.e. $\delta(G) \le n/2 \le \Delta(G)$, then
the best bound is $-n/2$.  Finally, if all of the slopes are positive,
i.e. $\delta(G) > n/2$, then at $x=1$ we obtain the bound
$\delta(G)-n$.

Noting that all of the diagonal entries of $A(G)$ are the same, we can
use Theorem~\ref{thm:halfGersh} and consider the $d_{ij}$.  If we
define \newcommand{\ddeg}{{\mathrm{ddeg}}} $\ddeg(i,j) =
1/2(\deg(i)+\deg(j))$ to be the average of the degrees of vertices $i$
and $j$, then
\begin{align*}
  d_{ij}(A,x) 
  &=\frac12 (d_i(A,x) + d_j(A,x)) = -\ddeg(i,j) + 2(\ddeg(i,j)-n)x.
\end{align*}
If we also define $\Delta_2(G)$ as the average of the two largest
degrees of $G$, and $\delta_2(G)$ as the average of the two smallest
degrees, then everything in the previous case analogously holds with $\Delta,\delta$ replaced by $\Delta_2,\delta_2$.

Finally, note that we can apply the Brauer method, in a similar
manner, directly.  Rewriting the above to obtain
\begin{equation*}
  R_i(A-x\om) = \deg(i) + (n-1-2\deg(i))x,
\end{equation*}
we have
\begin{equation*}
  f_{ij}(A,x) = -x-\sqrt{(\deg(i) + (n-1-2\deg(i))x)(\deg(j) + (n-1-2\deg(j))x)}
\end{equation*}
Again, we see that all of these functions are equal at $x=1/2$ and
again take value $-n/2$.  At $x=0$, this gives the unshifted Brauer
bound of $-\sqrt{\deg(i)\deg(j)}$ which is of course minimized at
$-\sqrt{\Delta(G)\Delta^2(G)}$.  At $x=1$, we obtain the bound
\begin{equation*}
  -1-\sqrt{((n-1)-\deg(i))((n-1)-\deg(j))}
\end{equation*}
which is minimized at $-1-\sqrt{(n-1-\delta(G))(n-1-\delta^2(G))}$.
As before, the first bound $-\sqrt{\Delta(G)\Delta^2(G)}$ is best for
graphs with small largest degree, the second bound
$-1-\sqrt{(n-1-\delta(G))(n-1-\delta^2(G))}$ is best for graphs with
large smallest degree, and the $-n/2$ bound is best for graphs with a
large separation of degrees.

One can check directly that these bounds are exact for the $(n,1)$-
and $(n,n/2-1)$-circulant graphs.  Of course when the graph is
regular, the Gershgorin and Brauer bounds coincide.  In
Figure~\ref{fig:adj} we plot numerical computations for a sample of
Erd\H{o}s--R\'enyi random graphs.  We consider the case with $N=20$
vertices and an edge probability of $p=0.9$ (we are considering the
case where the minimal degree is large, so the shifted bound is best).
We see that for a few graphs, this estimate is almost exact, and seems
to be not too bad for a large selection of graphs.  Note that the
unshifted bounds given by either Gershgorin or Brauer are much worse
in this case: the average degree of the vertices in these graphs is
$18$, and the average maximal degree is larger still
(see~\cite{Riordan.Selby.00} for more precise statements), so that the
unshifted methods would give bounds near $-18$, which is much further
from the actual values.

\begin{figure}[ht]
\begin{centering}
  \includegraphics[width=.9\textwidth]{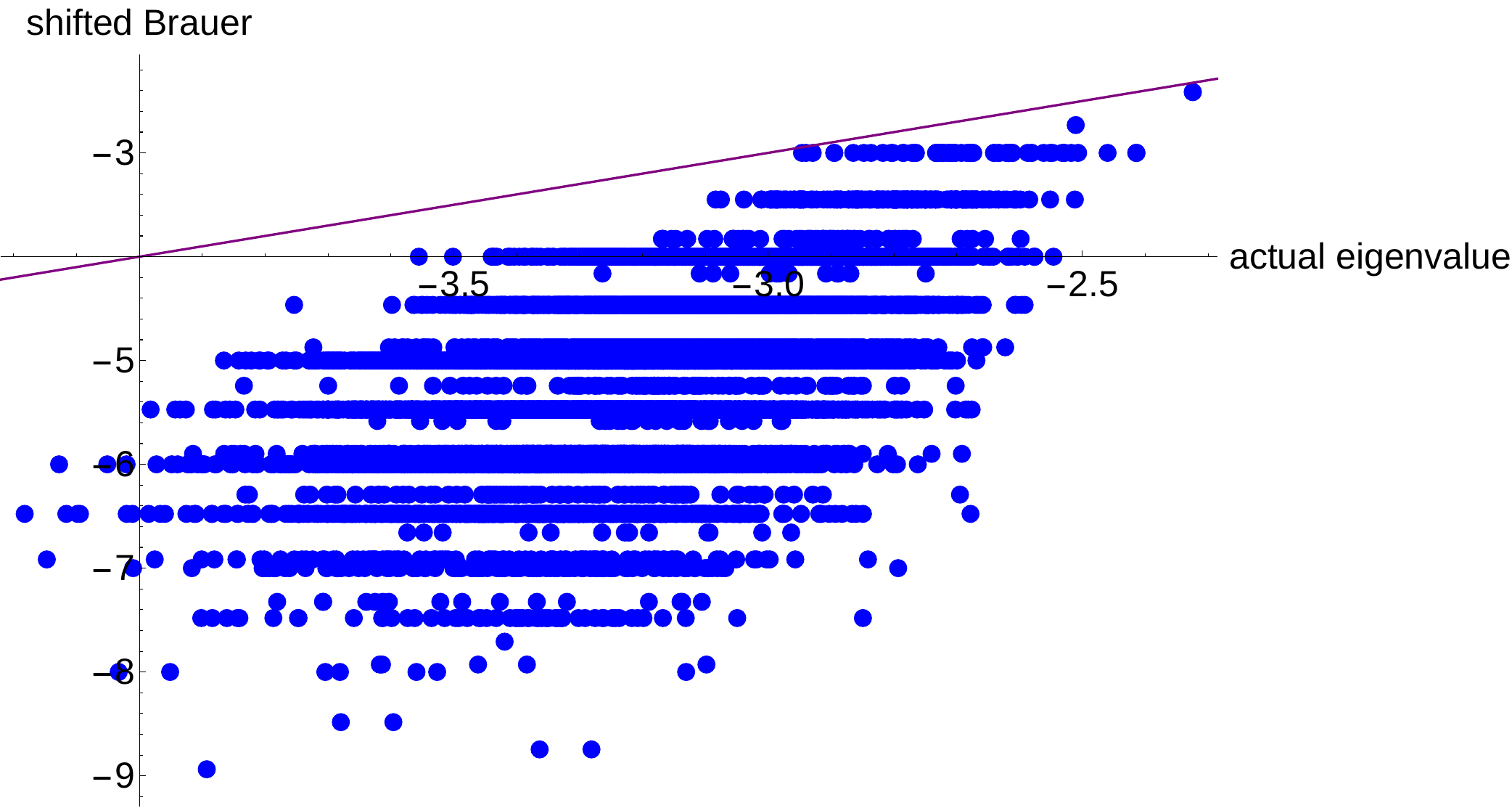}
  \caption{Plot of 10,000 samples of Erd\H{o}s--R\'enyi random graphs
    $G(N,p)$ with $N=20$ vertices and edge probability $p=0.9$.  For
    each random graph, we choose $A$ to be the adjacency matrix of the
    graph, and we are plotting $\lambda_1(A)$ on the $x$-axis and $\ls
    B A$ on the $y$-axis. }
  \label{fig:adj}
\end{centering}
\end{figure}

%\subsection{Markov chain}
%
%A (discrete-time, finite-state) Markov chain can be represented by a
%matrix with certain properties, which we define here.
%
%\begin{define}
%  An $n\times n$ {\bf stochastic matrix} is a matrix $P = \{p_{ij}\}$
%  with $p_{ij}\in[0,1]$ for all $i,j$; $\sum_j p_{ij} = 1$ for all
%  $i$.  A {\bf doubly stochastic matrix} is a stochastic matrix with
%  the further condition that $\sum_i p_{ij} = 1$ for all $j$.  Note
%  that any symmetric stochastic matrix is automatically doubly
%  stochastic.
%\end{define}
%
%Here we consider bounds on the eigenvalues of a symmetric stochastic
%matrix.  First note that if $P$ is stochastic, then $P{\bf 1} = {\bf
%  1}$, so that $1$ is always in the spectrum of $P$.  Note that
%symmetric matrices also have ${\bf 1}^\intercal P = {\bf
%  1}^\intercal$, which means that the invariant measure of a symmetric
%Markov chain is always the constant vector.  Under very mild
%conditions on the coefficients of $P$, the unit eigenvalue is simple;
%such a Markov Chain is called {\bf ergodic}. It follows from
%Perron--Frobenius that all of the other eigenvalues lie in the unit
%circle, so in the ergodic case this implies that $P^n \to \om$.  For
%Markov chains, the eigenvalues are typically ordered with decreasing
%modulus, so that $\lambda_1 = 1$ and $\lambda_2$ is the largest
%eigenvalue in modulus aside from the trivial one.
%
%
%
% First note that since the entries are nonnegative, it is
%natural to consider left shifted bounds on the matrix itself; however,
%when studying Markov chains we are more interested in right bounds, in
%a sense.

\subsection{Numerical comparisons of unshifted and shifted
  methods}\label{sec:randomnumerics}

Here we present the results of a numerical study comparing methods
studied here to the existing methods, where we compare the left bound
given by multiple methods for matrices with nonnegative coefficients.

We present the results of the numerical experiments in
Figures~\ref{fig:4} and~\ref{fig:sc}.  In each case, we considered
$1000$ $5\times 5$ matrices with random integer coefficients in the
range $0,\dots,10$. For each random matrix $A$, we compute the
unshifted Gershgorin bound $\lb G A$, the unshifted Brauer bound $\lb
B A$, the shifted Gershgorin bound $\ls G A$, and the actual smallest
eigenvalue.  In terms of computation, we computed all of these
directly except for $\ls G A$, and here we
used~\eqref{eq:singlefunction}.

In the left frame of each figure, we plot all four numbers for each
matrix.  Since these are all random samples, ordering is irrelevant,
and we have chosen to order them by the value of $\lb B A$ (yellow).
We have plotted $\lb G A$ in blue, and we confirm in this picture that
$\lb G A \le \lb B A$.  Finally, we plot $\ls GA$ in green and the
actual smallest eigenvalue $\lambda_1(A)$ in red.  We see that $\ls G
A > \lb B A$ for most samples, and is actually quite a bit better in
many cases.  We further compare $\lb B A$ and $\ls G A$ in
Figure~\ref{fig:sc}, where we have a scatterplot of the error given by
the two bounds: on the $x$-axis we plot $\lambda_1(A) - \ls G A$ (the
error given by shifted Gershgorin) and on the $y$-axis we plot
$\lambda_1(A) - \lb B A$ (the error given by unshifted Brauer).  We also
plot the line $y=x$ for comparison.  We again see that shifted
Gershgorin frequently beats unshifted Brauer, and typically by quite a
lot.  In fact, in 1000 samples we found that they were exactly the
same 4 times, and unshifted Brauer was better 21 times, so that shifted
Gershgorin gives a stronger estimate $97.5\%$ of the time.

We plot the same data for $10\times10$ matrices (again with integer
coefficients from $0,\dots,10$) in Figures~\ref{fig:410}
and~\ref{fig:sc10}.  We see that for the larger matrices, the
separation between the estimates and the actual eigenvalue grows, but
$\ls G A$ does much better than $\lb B A$.  In particular, out of 1000
samples, $\ls G A$ beats $\lb B A$ every time.

\begin{figure}[ht]
\begin{centering}
  \includegraphics[width=.8\textwidth]{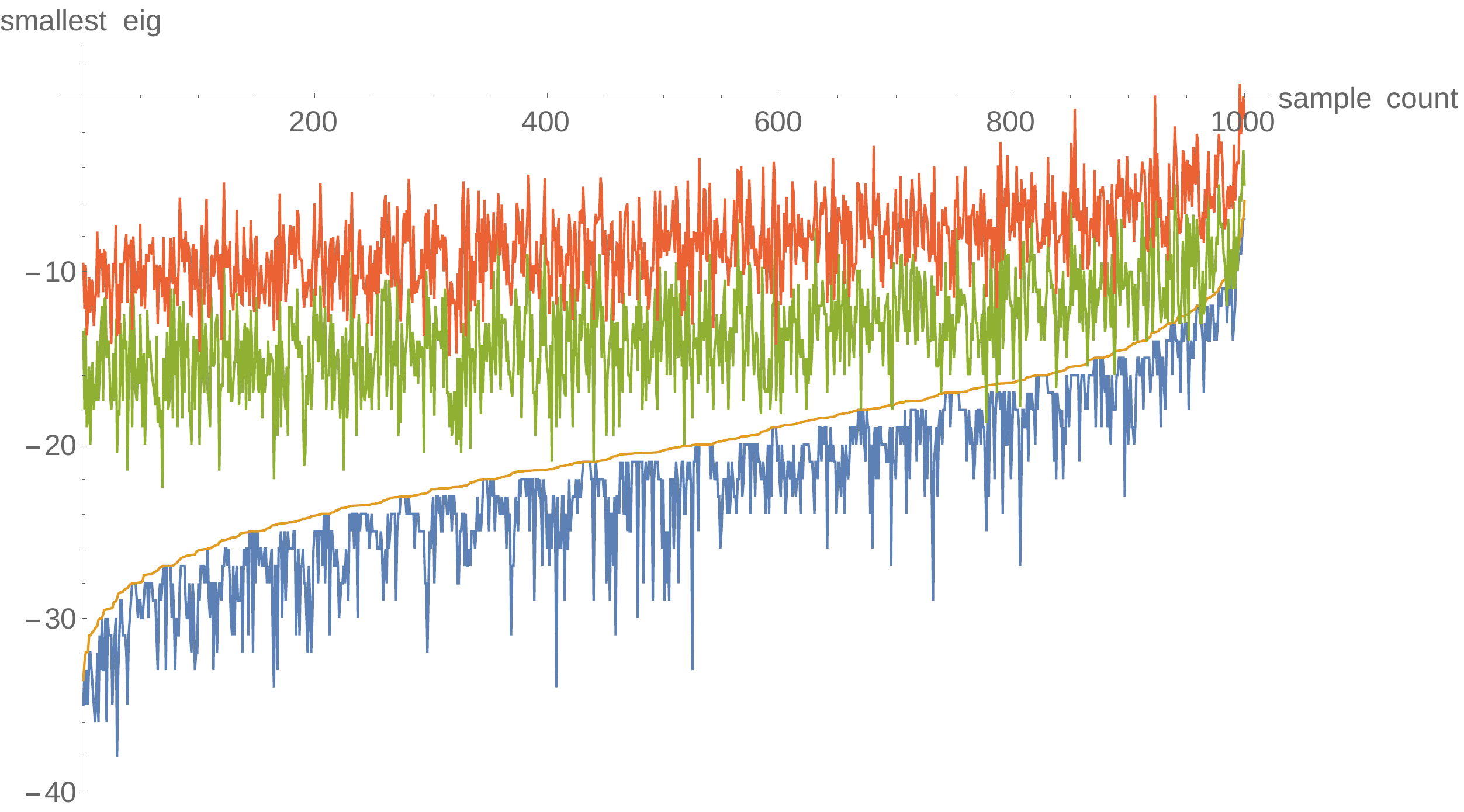}
  \caption{A calculation of four quantities for each of $1000$ random
    $5\times 5$ matrices.  We plot $\lb G A$ in blue, $\lb B A$ in
    yellow, $\ls G A$ in green, and $\lambda_1(A)$ in red.  These have
    been sorted by the value of $\lb B A$ for easier visualization.}
  \label{fig:4}
\end{centering}
\end{figure}

\begin{figure}[ht]
\begin{centering}
  \includegraphics[width=.8\textwidth]{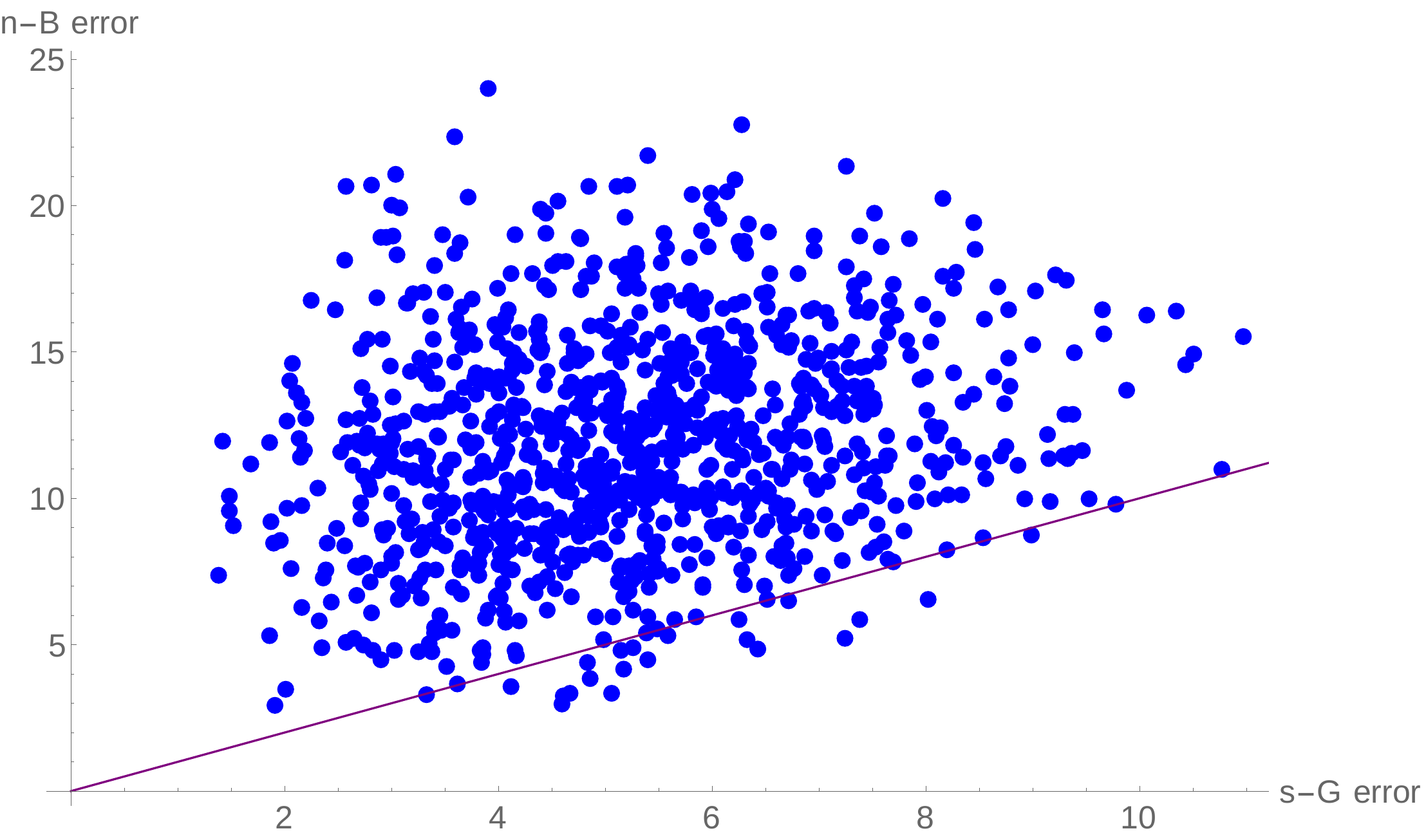}
  \caption{The same data as Figure~\ref{fig:4} where we plot
    $\lambda_1(A) - \ls G A$ on the $x$-axis and $\lambda_1(A)-\lb B
    A$ on the $y$-axis.}
  \label{fig:sc}
\end{centering}
\end{figure}

\begin{figure}[ht]
\begin{centering}
  \includegraphics[width=.8\textwidth]{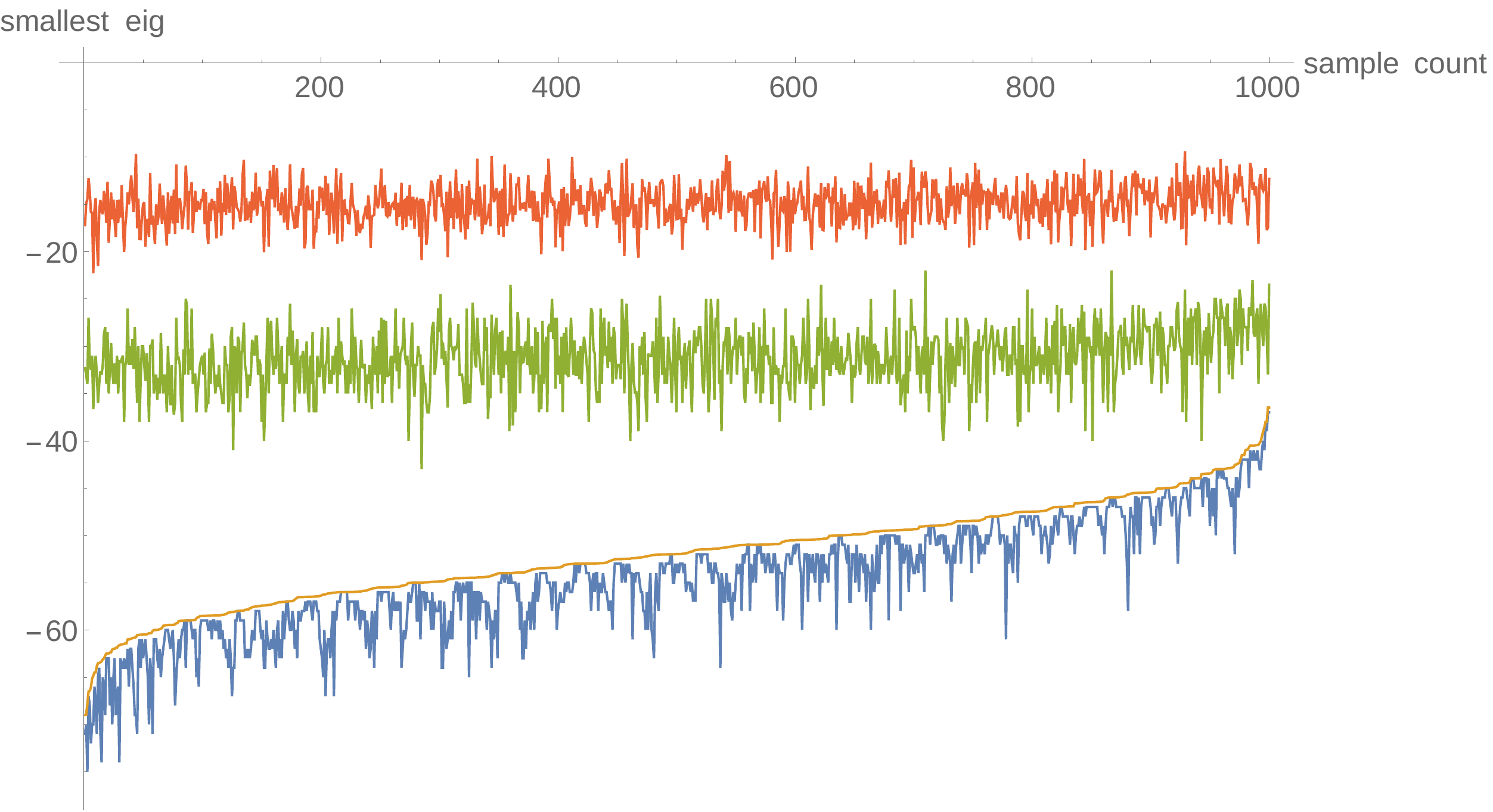}
  \caption{A calculation of four quantities for each of $1000$ random
    $10\times 10$ matrices.  We plot $\lb G A$ in blue, $\lb B A$ in
    yellow, $\ls G A$ in green, and $\lambda_1(A)$ in red.  These have
    been sorted by the value of $\lb B A$ for easier visualization.}
  \label{fig:410}
\end{centering}
\end{figure}

\begin{figure}[ht]
\begin{centering}
  \includegraphics[width=.8\textwidth]{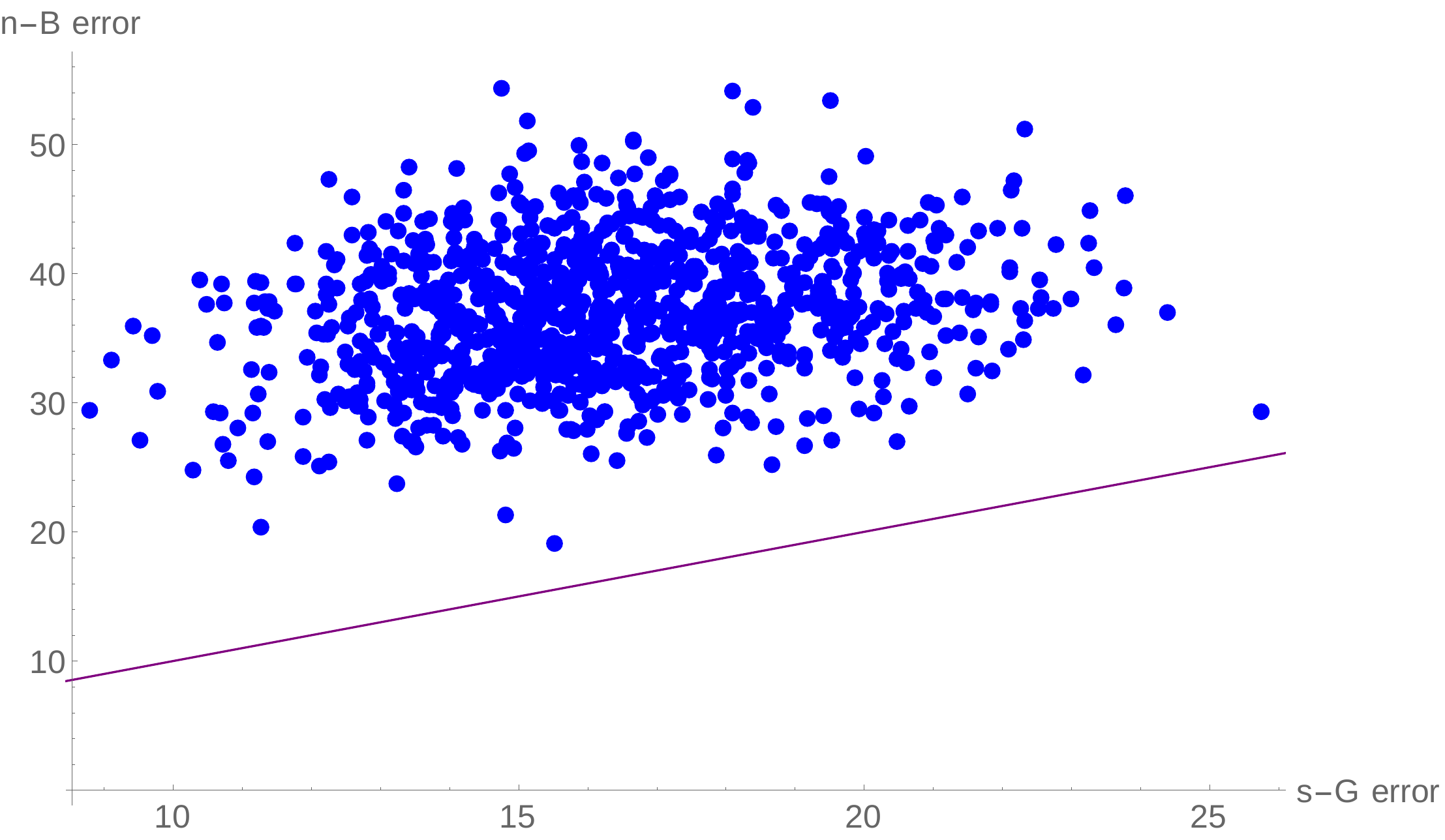}
  \caption{The same data as Figure~\ref{fig:410} where we plot
    $\lambda_1(A) - \ls G A$ on the $x$-axis and $\lambda_1(A)-\lb B
    A$ on the $y$-axis.}
  \label{fig:sc10}
\end{centering}
\end{figure}

\subsection{ Domain of positive definiteness}\label{sec:pdnumerics}

One other way to view the problem of positive definiteness is as
follows.  Let us imagine that a matrix is given where the off-diagonal
entries are prescribed, but the diagonal entries are free.  We could
then ask what condition on the diagonal entries guarantees that the
matrix be positive definite?

First note that as we send the diagonal entries to infinity, the
matrix is guaranteed positive definite, simply because of the
Gershgorin result: as long as $a_{ii} > \sum_{j\neq i}\av{a_{ij}}$ for
all $i$, then the matrix is positive definite.  Thus we obtain at
worst an unbounded box from unshifted Gershgorin.

From the fact that the shifted Gershgorin estimates are always
piecewise linear, we would expect that the region from this is always
an ``infinite polytope'', i.e. an unbounded intersection of finitely
many half-planes that is a superset of the unshifted Gershgorin box.  We
prove this here.

\begin{prop}\label{cor:polytope}
  Let $v$ be a vector of length $(n^2-n)/2$, and let $\mathcal{M}(v)$
  be the set of all matrices with $v$ as the off-diagonal terms
  (ordered in the obvious manner).  Let us define $\mathcal{S}(v)$ as
  the subset of $\mathcal{M}(v)$ where $A\in\mathcal{S}(v)$ if $\ls G
  A \ge 0$, and let $\mathcal{D}(v)$ to be the set of all diagonals of
  matrices in $\mathcal{S}(v)$.  Then the set $\mathcal{D}(v)$ is an
  unbounded domain defined by the intersection of finitely many
  half-planes that contains a box of the form $a_{ii} \ge \beta_i$,
  and as such can be written $\cap_\ell \{x:a_\ell\cdot x\ge 0\}$
  where the entries of $a_{\ell}$ are nonnegative.
\end{prop}

\begin{proof}
  We have from Theorem~\ref{thm:gersh} that
\begin{equation*}
  \ls G A = \sup_{x\ge 0} \min_{k=1}^n (r_k x + s_k).
\end{equation*}
This means that $\ls G A\ge 0$ iff there exists $x \ge 0$ such that
$r_k x + s_k\ge 0$ for all $k = 1,\dots, n$.  This is the same as saying that
\begin{equation*}
  \bigcap _{k=1}^n \{x\ge0 : r_k x +s_k \ge 0\} \neq \emptyset.
\end{equation*}
Under the condition that $s_{n/2}\ge 0$ (or under no additional condition if $n$ is odd), this is the same as 
\begin{equation*}
   \bigcap_{k < n/2}\{x\ge0: x\ge -s_k/r_k\}\cap \bigcap_{k > n/2}\{x\ge0: x\le -s_k/r_k\}\neq \emptyset.
\end{equation*}
This reduces to the two conditions
\begin{equation*}
  \max_{k<n/2} -s_k/r_k \le \min_{k>n/2} - s_k/r_k,\quad \min_{k>n/2} - s_k/r_k \ge 0,
\end{equation*}
which is clearly the intersection of half-planes in $s_k$.  Fixing the
off-diagonal elements of $A$ fixes $\delta_{i,k}$, and then $s_k$ is
the minimum of functions linear in the diagonal elements of $A$.
Finally we know that $\ls G A \le \lb G A$, so that any matrix
satisfying $a_{ii} \ge R_i(A)$ is clearly in $\mathcal{S}(v)$.
\end{proof}

It follows from this proposition that shifted Gershgorin cannot be
optimal.  The boundary of the set $\mathcal{D}(v)$ must be defined by
an $n$th degree polynomial in the $a_{ii}$, since this boundary is
given by an eigenvalue passing through zero and thus must be a subset
of $\det(A) = 0$.

\begin{exa}\label{exa:region}
  Let us consider a $3\times 3$ example for this region, and to make
  regions easier to plot let us only let two of the diagonal entries
  vary, namely
\begin{equation}\label{eq:defofAregion}
  A=\left(\begin{array}{ccc} y&a&b \\ a&z&c \\b&c&d\end{array}\right),
\end{equation}
where $a,b,c,d$ are fixed parameters and $y,z$ are variable.  From
Proposition~\ref{cor:polytope}, the domain $\{(y,z):\ls G A \ge 0\}$
is an intersection of half-planes; the domain $\{(y,z):\lb G A \ge
0\}$ is given by a box, and of course the domain $\{(y,z): A \ge 0\}$
is given by a quadratic equation in $y,z$.

Let us consider one example: choose $a=2,b=1,c=2,d=4$, which is shown
by the second frame in Figure~\ref{fig:regions}.  We can compute the
conditions on $y,z$ as follows.  We have 
\begin{equation*}
  s_1 = \min(1,y-3,z-4),\quad s_2 = \min(3,y-1,z),\quad s_3 = \min(7,y+3,z+4),
\end{equation*}
and the domain is defined by $s_2\ge 0,s_3\ge 0,s_2\ge -s_1, s_3 \ge -
3s_1$.  Writing all of this out, we obtain the conditions
\begin{equation*}
  y \ge 2, \quad z\ge 2,\quad y+z\ge 5.
\end{equation*}
Note that unshifted Gershgorin would give the conditions $y\ge 3, z\ge
4$.  The exact boundary condition from the determinant is $4yz-4y-z-8
= 0$. We plot all of these sets in Figure~\ref{fig:regions} for this
and other choices of parameters.  Note that the shifted region
sometimes shares boundary with the exact region, but being piecewise
linear this is the best that can be done.
\end{exa}

\begin{figure}[ht]
\begin{centering}
  \includegraphics[width=.3\textwidth]{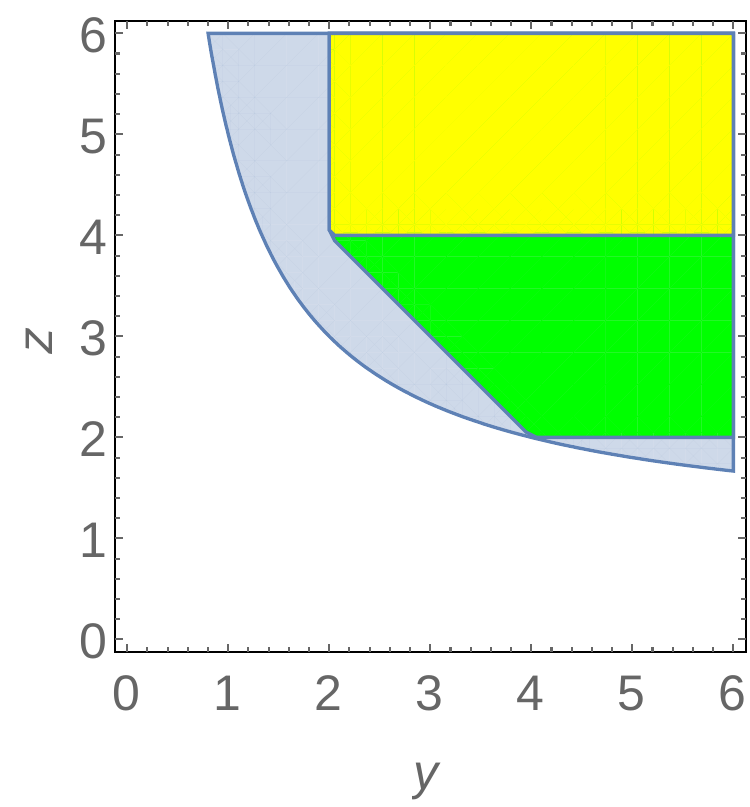}
  \includegraphics[width=.3\textwidth]{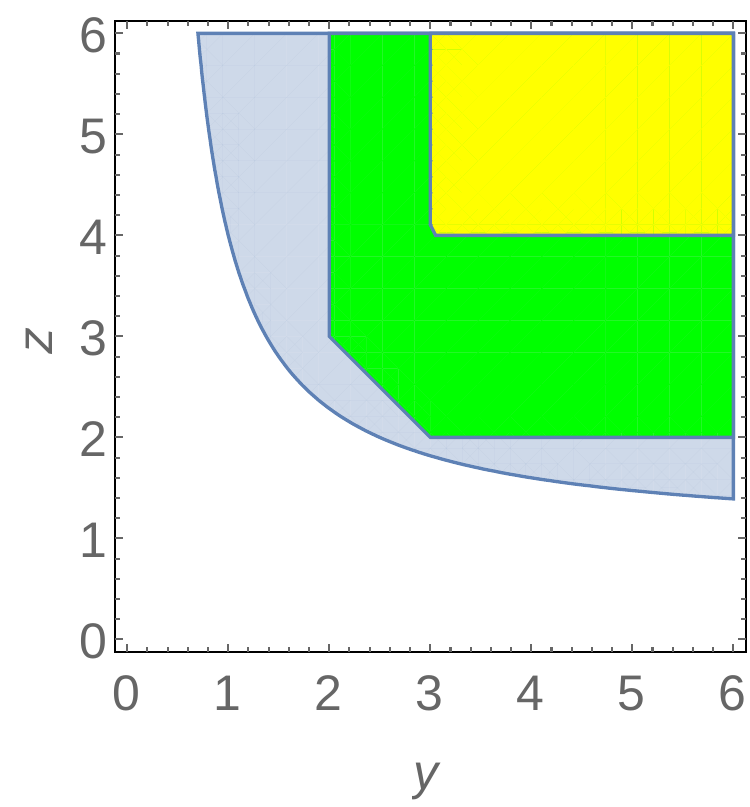}
  \includegraphics[width=.3\textwidth]{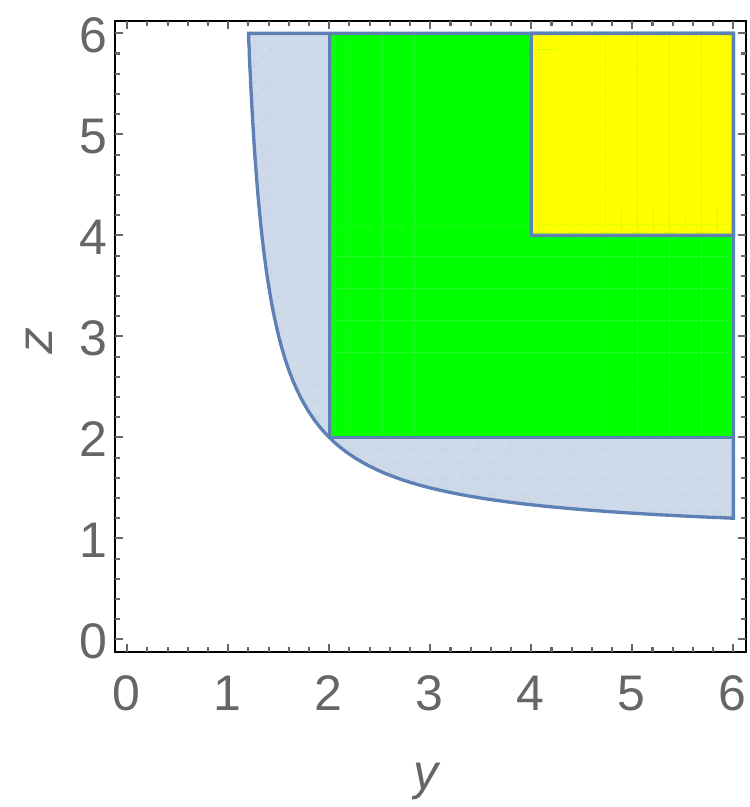}
  \caption{ Positive-definiteness for the matrix $A$ defined
    in~\eqref{eq:defofAregion}.  In each case, the actual region of
    positive-definiteness is plotted in gray, the region guaranteed by
    shifted Gershgorin given in green, and the region guaranteed by
    unshifted Gershgorin in yellow.  The three pictures correspond to
    the parameter choices $(2,0,2,4),(2,1,2,4),(2,2,2,4)$
    respectively.}
  \label{fig:regions}
\end{centering}
\end{figure}

\subsection{Spread of the off-diagonal terms}\label{sec:spread}

We might ask the question, given the $R_i(A)$, what are the best and
worst choices of off-diagonal entries for the bound $\lb G A$.
Intuitively, it seems clear that concentrating all of the mass in a
single off-diagonal element is ``worst'', whereas spreading it as much
as possible is ``best''.  In short, high off-diagonal variance is bad
for the shifted estimates.

\begin{prop}\label{prop:bestRi}
  Let $v(N-1,\rho)$ be the set of all vectors of length $N-1$ with
  nonnegative entries that sum to $\rho$.  We can then consider the
  quantities
  \begin{align*}
    \underline{d}(N-1,\rho) &= \inf_{y\in v(N-1,\rho)} \sup_{x\ge 0} \l(\rho-x-\sum_{k=1}^{n-1}\av{y_k-x}\r),\\
    \overline{d}(N-1,\rho) &= \sup_{y\in v(N-1,\rho)} \sup_{x\ge 0} \l(\rho-x-\sum_{k=1}^{n-1}\av{y_k-x}\r).
  \end{align*}
  Then $\underline{d}(N-1,\rho) = 0$ and is attained at any $v$ with
  more than half of the entries equal to zero, and
  $\overline{d}(N-1,\rho) = \rho(N-2)/(N-1)$ and is attained at the
  vector whose entries are equal to $\rho/(N-1)$.
\end{prop}

\begin{proof}
  The claim about $\underline{d}(N-1,\rho)$ is straightforward.  If
  more than half of the entries are zero, then it is easy to see that
  the function is decreasing in $x$ and its supremum is attained at
  $x=0$, giving zero.  Moreover, this infimum cannot be less than
  zero.

  For $\overline{d}(N-1,\rho)$, we first show that for any $y$ with $\sum_k y_k = \rho/(N-1)$,
\begin{equation*}
  \min_k y_k + \sum_{\ell\neq k} \av{y_k-y_\ell} \ge \frac\rho{N-1}.
\end{equation*}
First note that if all $y_k = \rho/(N-1)$ then equality is satisfied.
Assume that $y$ is not a constant vector.  If $y_k \ge \rho/(N-1)$
then the inequality is trivial, so assume $y_k < \rho/(N-1)$.  Since
the average of the entries of $y$ is $\rho/(N-1)$, at least one
$y_{\ell^*} > \rho/(N-1)$.  Writing $\alpha = \rho/(N-1)-y_k$, we have
$\av{y_k-y_{\ell^*}} \ge \alpha$ and again the inequality follows.
Finally, note that since the quantity is linear away from the points
$y_k$, we have
\begin{equation*}
  \sup_{x\ge 0} \rho-x-\sum_{k=1}^{n-1}\av{y_k-x} = \max_k \rho-y_k -\sum_{\ell \neq k} \av{y_\ell-y_k},
\end{equation*}
and we are done.
\end{proof}

This proposition verifies the intuition that ``spreading'' the mass in
the off-diagonal terms improves the shifted Gershgorin bound while not
changing the unshifted Gershgorin bound at all.  For example, let us
consider all symmetric matrices with nonnegative coefficients with all
diagonal entries equal to $\rho$ and all off-diagonal sums equal to
$\rho$.  Then for any $A$ in this class, $\lb G A = 0$.  However, $\ls
G A$ can be as large as $\rho(N-2)/(N-1)$ if all of the off-diagonal
entries are the same.  Moreover, it is not hard to see that if the
off-diagonal entries differ by no more than $\epsilon$, then $\ls G A
= \rho(N-2)/(N-1)-O(\epsilon)$.  In this sense, spreading the
off-diagonal mass as evenly as possible helps the shifted Gershgorin
bound.

Conversely, concentrating the off-diagonal mass makes the shifted
Gershgorin bound worse, and, for example, if we consider the case
where each row has one entry equal to $\rho$ and the rest zero, then
$\ls G A = \lb G A=0$.  Moreover, note that this bound cannot be
improved without further information, since it is sharp for
$\rho(I+P)$ where $P$ is any permutation matrix with a two-cycle.

\section{Conclusions}\label{sec:outtro}

We have shown a few cases where shifting the entries of a matrix leads
to better bounds on the spectrum of the matrix.  We can think of these
results as a flip of the standard Perron--Frobenius theorem; the
method of this paper gives good bounds on the lowest eigenvalue of a
positive matrix rather than the largest.

We might ask whether the results considered here are optimal, in the
sense of obtaining a spectral bound from a linear program in terms of
the entries of a matrix.  Of course the last condition is necessary:
writing down the characteristic polynomial of a matrix and solving it
gives excellent bounds on the eigenvalues, but this is a notoriously
challenging approach.  Most of the existing methods mentioned above
(again see~\cite{Varga.book} for a comprehensive review) use some
function of the off-diagonal terms, e.g. their sum, or their sum minus
one distinguished element, or perhaps two subsums of the entries. This
is true of all of the methods mentioned above, as is typical of
results of this type.  The shifted Gershgorin method uses the actual
off-diagonal entries, and as we have shown above, its efficacy is a
function as much of the variance of these off-diagonal terms as their
sum.  As we showed numerically in Section~\ref{sec:randomnumerics},
shifted Gershgorin beats even a nonlinear method very often in a large
family of matrices (to be fair, it was a family designed to work well
with shifted Gershgorin, but in that class it works very well).
Although checking against every existing method is beyond the scope of
this paper, it seems likely that one can construct examples where
shifted Gershgorin beats any existing method that uses less
information about the off-diagonal terms than their actual values.

We also might ask whether this technique can be improved by shifting
by a matrix other than ${\bf 1}$ (cf.~\cite{Higham.Tisseur.03} for a
similar perspective on a more complicated bounding problem).  Of
course, if we choose any positive semidefinite matrix $C$, write $A =
B + xC$, and optimize $x$ to obtain the best bounds on the matrix $B$,
this will also give estimates on the eigenvalues of $A$.  It surely is
true that for a given matrix $A$, there could be a choice of $C$ that
does better than ${\bf 1}$ in this regard by exploiting some feature
of the original matrix $A$. It seems unlikely that there could be a
matrix $C$ that generally beats ${\bf 1}$, especially in light of the
results of Section~\ref{sec:spread}.  Also, there would be the
challenge of knowing $C$ was positive semi-definite in the first
place: it seems unlikely that one could bootstrap this method by
writing $C = D + y {\bf 1}$, since this would give $A = (B+xD) +
xy{\bf 1}$, and one is still optimizing over multiples of ${\bf 1}$.
However, if there is some known structure of $A$ that matches the
structure of a known semi-definite $C$ (e.g. $C$ is a covariance
matrix or more generally a Gram matrix) then this might allow us to
obtain tighter bounds on the spectrum of $A$.

\section*{Acknowledgments}
The author would like to thank Jared Bronski for useful discussions.
%\bibliographystyle{unsrt} 
%\bibliography{deville,gersh}

\end{document}